\documentclass[pdflatex]{sn-jnl}

\usepackage{graphicx}%
\usepackage{multirow}%
\usepackage{amsmath,amssymb,amsfonts}%
\usepackage{amsthm}%
\usepackage{mathrsfs}%
\usepackage[title]{appendix}%
\usepackage{xcolor}%
\usepackage{textcomp}%
\usepackage{manyfoot}%
\usepackage{booktabs}%
\usepackage{algorithm}%
\usepackage{algorithmicx}%
\usepackage{algpseudocode}%
\usepackage{listings}%

\usepackage{hyperref} 
\usepackage{pgfplots} 
\usepackage[numbers, sort&compress]{natbib}
\pgfplotsset{compat=1.18} 
\usetikzlibrary{patterns} 
\usetikzlibrary{patterns, arrows.meta}

\theoremstyle{thmstyleone}%
\newtheorem{theorem}{Theorem}[section] 
\newtheorem{lemma}[theorem]{Lemma}

\theoremstyle{thmstyletwo}%
\newtheorem{remark}[theorem]{Remark}

\theoremstyle{thmstylethree}%
\newtheorem{definition}[theorem]{Definition}

\begin{document}

\title[Shinbrot Type Criteria for Energy Conservation]{Shinbrot Type Criteria for Energy Conservation of the Compressible Navier-Stokes Equations}


\author[1]{\fnm{Ruxuan} \sur{Chen}}\email{rc1c24@soton.ac.uk}

\author*[2]{\fnm{Qi} \sur{Zhang}}\email{qizhangmath@hrbeu.edu.cn}

\author[2]{\fnm{Zhikang} \sur{Zhang}}\email{zhangzhikang23@163.com}

\author[2]{\fnm{Xiongbo} \sur{Zheng}}\email{zhengxiongbo@hrbeu.edu.cn}

\affil[1]{\orgdiv{Southampton Ocean Engineering Joint Institute at HEU}, \orgname{Harbin Engineering University}, \orgaddress{\street{Nantong Street}, \city{Harbin}, \postcode{
			150001}, \state{Heilongjiang Province}, \country{P. R. China}}}

\affil[2]{\orgdiv{College of Mathematical Sciences}, \orgname{Harbin Engineering University}, \orgaddress{\street{Nantong Street}, \city{Harbin}, \postcode{
			150001}, \state{Heilongjiang Province}, \country{P. R. China}}}


\abstract{We prove that weak solutions to the compressible Navier–Stokes equations satisfy the energy equality under a Shinbrot-type regularity criterion. Our method applies to the fluids with both constant and degenerate viscosity and relies on a novel weak-type commutator estimate. These criterion are strictly weaker than those required in prior works \cite[Arch. Ration. Mech. Anal., 225 (2017)]{yu2017energy} and \cite[SIAM J. Math. Anal. 52 (2020)]{chen2020energy}.
	}


\maketitle

\section{Introduction}

The compressible Navier-Stokes equations are a fundamental model in fluid dynamics, describing the motion of a viscous compressible fluid. In this paper, we mainly study the energy conservation of the following system in a bounded domain $\Omega \subset \mathbb{R}^N$ ($N = 2, 3$):
\begin{equation}\label{eq:comNS}
    \begin{split}
        (\rho u)_t + \text{div}(\rho u \otimes u) - \mu \Delta u - (\mu+\lambda) \nabla \text{div} \, u + \nabla P &= 0,  \\
        \rho_t + \text{div}(\rho u) &= 0, 
    \end{split}
\end{equation}
with initial data
\begin{equation}
    \rho|_{t=0} = \rho_0(x), \quad \rho u|_{t=0} = \rho_0 u_0, \label{con:initial}
\end{equation}
where $P = \rho^\gamma$, $\gamma > 1$, is the pressure, $\rho$ is the density, $u$ is the velocity, and the viscosity coefficients satisfy $\mu > 0$, $2\mu + N\lambda \geq 0$. Here we define $u_0 = 0$ on the set $\{x|\rho_0(x) = 0\}$. However, the weak solutions of these equations do not, in general, satisfy the energy equalities, i.e.
\begin{equation}
    \begin{aligned}
    	&\int_{\Omega}\left(\frac{1}{2} \rho_0\left|u_0\right|^2+\frac{\rho_0^\gamma}{\gamma-1}\right) \mathrm{d} x-\int_{\Omega}\left(\frac{1}{2} \rho|u|^2+\frac{\rho^\gamma}{\gamma-1}\right) \mathrm{d} x\\
    	=&\int_0^T \int_{\Omega}\left( \mu|\nabla u|^2+(\mu+\lambda)|\operatorname{div} u|^2 \right)\mathrm{d} x \mathrm{d} t. \label{eq:energy_comNS}
    \end{aligned}
\end{equation}
This phenomenon, along with incompressible cases, has drawn attention and sparked interest in exploring the relationship between energy conservation and the regularity of weak solutions \cite{buckmaster2016dissipative,constantin1994onsager,feireisl2017regularity,isett2018proof}.

In fact, this subject originated from the famous Onsager’s conjecture for incompressible Euler flows \cite{onsager1949statistical}. This conjecture states that 1/3 should be the threshold of Hölder continuous exponent for energy conservation. Consequently, \cite{constantin1994onsager,eyink1994energy,cheskidov2008energy} obtained the energy equality with Besov regularity greater than 1/3 and \cite{buckmaster2015anomalous,buckmaster2016dissipative,isett2018proof} constructed the non-conservation solutions with Hölder continuity smaller than 1/3. These significant advances almost reached the threshold.

In the context of the Navier–Stokes equations, the pioneering studies were done by Lions \cite{lions1960regularite} and Serrin \cite{serrininitial}. Lions proved that a weak solution $u$ to a incompressible fluids conserves its energy provided $u \in L^{4}_{t,x}$ and Serrin gave a dimension-dependent condition
\begin{equation}\label{con:serrin}
    u \in L^{p}_{t}L^{q}_{x}\quad\mbox{for}\quad \frac{2}{p}+\frac{N}{q}\leq 1,\quad q > N.
\end{equation}
Later, Shinbrot removed the dimensional dependence in \cite{shinbrot1974energy} and proved the same conclusion if
\begin{equation}\label{con:shinbrot}
    \frac{2}{p}+\frac{2}{q}\leq 1,\quad q\geq 4,
\end{equation}
which improved the previous work. Meanwhile, for $3\leq q<4$, the following criteria
\begin{equation}\label{con:shinbrot2}
    \frac{1}{p}+\frac{3}{q}\leq 1
\end{equation}
could also yield energy conservation due to the embedding relationship \cite{da2020shinbrot}. Furthermore, a range of new types of conditions have been obtained recently, see \cite{drivas2019onsager,cheskidov2020energy,leslie2018conditions}.

As for the compressible systems, corresponding results are more recent and the analysis becomes inherently more complex. Within the framework of Onsager's theory \cite{onsager1949statistical}, Drivas and Eyink \cite{drivas2018onsager1,eyink2018cascades} adapted methods from the incompressible fluids \cite{constantin1994onsager,duchon2000inertial,eyink1994energy} to the compressible Euler equations and obtain necessary conditions for the energy dissipative anomalies of turbulent solutions. Meanwhile, sufficient conditions for energy conservation in terms of Besov regularity for the compressible Euler fluids with a $C^{\alpha}$ ($1<\alpha\leq2$) pressure law $p(\rho)$ were later established by Feireisl et al. \cite{feireisl2017regularity} and Akramov et al. \cite{akramov2020energy}.

Regarding the energy equality of compressible Navier-Stokes \eqref{eq:comNS}, the pioneering breakthrough was made by Yu \cite{yu2017energy}. He proved the equalities \eqref{eq:energy_comNS} and \eqref{eq:energy_comNSdeg} with the conditions that the weak solution $(\rho,u)$ satisfies
\begin{equation}\label{con:Yu}
    u \in L^{p}_{t}(0,T;L^{q}_{x}(\mathbb{T}^{N}))\quad\mbox{for}\quad \frac{1}{p}+\frac{1}{q}\leq \frac{5}{12},\quad q\geq 6
\end{equation}
and
\begin{equation}\label{con:genhao rou}
    \sqrt{\rho} \in L^\infty (0, T; H^1(\mathbb{T}^{N})).
\end{equation}
Later, Chen et al. \cite{chen2020energy} treated the  boundary effect to investigate energy conservation on a more general bounded domain with no-slip boundary condition, when
\begin{equation}\label{con:Chen et al.}
    u \in L^{p}_{t}(0,T;L^{q}_{x}(\Omega))\quad\mbox{for}\quad p\geq4,\quad q\geq 6,
\end{equation}
which is just the end point of \eqref{con:Yu}. Particularly, Ye et al. \cite{ye2022energy} have reached $L^{4}L^{4}$ criteria for the compressible Navier-Stokes equations, while they provided that $\nabla\sqrt{\rho}\in L^{4}L^{4}$ instead of $L^{\infty}L^{2}$. However, the $L^{\infty}H^{1}$ regularity of $\sqrt{\rho}$ is of independent significance, as the Bresch-Desjardins entropy for \eqref{eq:comNSdeg} implies such an estimate \cite{vasseur2016existence}. In conclusion, a gap persists between these results and the Shinbrot criteria, due to the complexity of the compressible system.

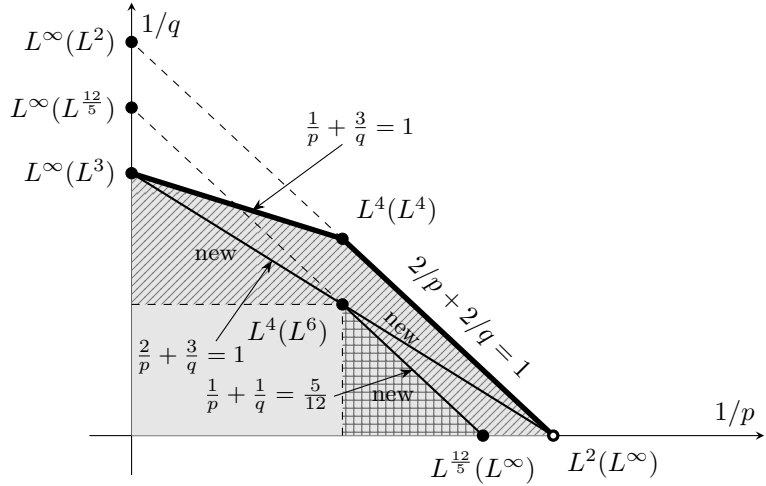
\begin{figure}[htbp]
  \centering 
  \begin{tikzpicture}
    \begin{axis}[
        width=0.8\textwidth, 
        height=0.6\textwidth, 
        xlabel={$1/p$},
        xlabel style={
        at={(rel axis cs:0,1)},
        anchor=west, 
        xshift=15pt, 
        yshift=-15pt 
    },
        ylabel={$1/q$},
        ylabel style={
        at={(0,0.55)},
        anchor=south, 
        xshift=10pt, 
        yshift=10pt, 
        rotate=0 
    },
        xmin=-0.05, xmax=0.75, ymin=-0.05, ymax=0.55,
        axis lines=middle,
        xtick={0},
        ytick={0},
        grid=major,
        legend pos=north east,
        clip=false,
    ]
      \fill[gray!20] (axis cs: 0,0) -- (axis cs: 0.5,0) -- (axis cs: 0.25,0.25) -- (axis cs: 0,1/3) -- cycle;
      \fill[pattern=north east lines, pattern color=black!50] (axis cs: 5/12,0) -- (axis cs: 0.25,1/6) -- (axis cs: 0,1/6) -- (axis cs: 0,1/3) -- (axis cs: 0.25,0.25) -- (axis cs: 0.5,0) -- cycle;
      \fill[pattern=grid, pattern color=black!60] (axis cs: 5/12,0) -- (axis cs: 0.25,1/6) -- (axis cs: 0.25,0) -- cycle;

      \node[draw, circle, fill=black, inner sep=1.5pt, label=below right:{$L^2(L^\infty)$}] at (axis cs: 0.5,0) {};
      \node[draw, circle, fill=black, inner sep=1.5pt, label=below left:{$L^4(L^6)$}] at (axis cs: 0.25,1/6) {};
      \node[draw, circle, fill=black, inner sep=1.5pt, label=above right:{$L^4(L^4)$}] at (axis cs: 0.25,0.25) {};
      \node[draw, circle, fill=black, inner sep=1.5pt, label=below:{$L^{\frac{12}{5}}(L^\infty)$}] at (axis cs: 5/12,0) {};
      \node[draw, circle, fill=black, inner sep=1.5pt, label=left:{$L^{\infty}(L^3)$}] at (axis cs: 0,1/3) {}; 
      \node[draw, circle, fill=black, inner sep=1.5pt, label=left:{$L^{\infty}(L^\frac{12}{5})$}] at (axis cs: 0,5/12) {};
      \node[draw, circle, fill=black, inner sep=1.5pt, label=left:{$L^{\infty}(L^2)$}] at (axis cs: 0,0.5) {};
      
      \draw[thick] (axis cs: 0,1/3) -- (axis cs: 0.5,0);
      \draw[line width=2pt] (axis cs: 0,1/3) -- (axis cs: 0.25,0.25);
      \draw[thick] (axis cs: 0.25,1/6) -- (axis cs: 5/12,0);
      \draw[line width=2pt] (axis cs: 0.25,0.25) -- (axis cs: 0.5,0);
      \draw[dashed] (axis cs: 0.25,0.25) -- (axis cs: 0,0.5);
      \draw[dashed] (axis cs: 0.25,1/6) -- (axis cs: 0,5/12);
      \draw[dashed] (axis cs: 0,1/6) -- (axis cs: 0.25,1/6);
      \draw[dashed] (axis cs: 0.25,0) -- (axis cs: 0.25,1/6);
      \draw[-Stealth] (0.1,0.12) -- (1/6,2/9);
      \draw[-Stealth] (0.24,0.05) -- (1/3,1/12);
      \draw[-Stealth] (0.25,0.37) -- (1/7,2/7);

      \node[rotate=-45, font=\normalsize] at (axis cs: 0.4, 0.15) {$2/p + 2/q = 1$};
      \node[rotate=0, font=\small] at (axis cs: 0.07, 0.1) {$\frac{2}{p} + \frac{3}{q} = 1$};
      \node[rotate=0, font=\small] at (axis cs: 0.16, 0.05) {$\frac{1}{p} + \frac{1}{q} = \frac{5}{12}$};
      \node[rotate=0, font=\small] at (axis cs: 0.27, 0.39) {$\frac{1}{p} + \frac{3}{q} = 1$};
      \node[rotate=-45, font=\small] at (axis cs: 0.32, 0.135) {new};
      \node[rotate=0, font=\small] at (axis cs: 0.1, 0.23) {new};
      \node[rotate=0, font=\small] at (axis cs: 0.31, 0.05) {new};
      \draw[fill=white, line width=0.8pt] (axis cs: 0.5,0) circle (1.5pt);

    \end{axis}
  \end{tikzpicture}

  \caption{The shaded region illustrates the exponent ranges covered by Theorem \ref{theoremcomNS}.The dashed and grid subregions represent new ranges of exponents not obtained before in the literature, where the grid area only corresponds to bounded domain with no-slip boundary and the dashed area corresponds to both bounded domain and torus.} 
  \label{fig:embedding-diagram} 
\end{figure}

In fact, temporal derivative term $\partial_{t}(\rho u)$ is the main challenges in filling this gap. It requests a spatiotemporal mollification of the equation, not a pure spatial mollification of the incompressible system. Then, new errors associated to temporal mollification have emerged. In the proofs of \cite{chen2020energy,ye2022energy,yu2017energy}, the energy flux from $\partial_{t}(\rho u)$ was handled by a commutator estimate providing the integrability of $\partial_{t}\rho$. Particularly, it is the mass equation which allows us to 'transform' the spatial regularity of $u$ into the temporal regularity of $\rho$. However, this strategy meets some restriction. One need to assume that the velocity $u$ complies with the Serrin type criteria at least, not the Lions or Shinbrot type criteria.

\subsection{Methodology}

The present paper aims to bridge the gap between known research (\eqref{con:Yu} and \eqref{con:Chen et al.}) and the Shinbrot type criteria (\eqref{con:shinbrot} and \eqref{con:shinbrot2}) on a periodic domain $\mathbb{T}^{d}$ or a bounded domain with $C^{1}$ boundary. For this purpose, we shall reorganize the framework and introduce a new mollification and estimate. These approaches allow us to deal with the energy flux caused by the nonlinear terms $(\rho u)_t$ under lower integrability conditions. In the following paragraphs, we briefly describe the ideas of our method.

\emph{Framework and mollification.} The proof framework developed in this paper consists of two main steps: first, establishing the local energy equality, and then extending it to a global version. In the first step, we derive the local energy equality using the specific test function $(\varphi u_{\varepsilon}^{\varepsilon})_{\varepsilon}^{\varepsilon}$, introduced in Section 2, and passing to the limit as $\varepsilon \to 0$. By structuring the mollifier in this manner, we can separate errors present in different directions. In particular, we establish the vanishing of the temporal error via our weak-type temporal commutator estimate, which improves upon the classical treatment. Meanwhile, the other terms, such as $\text{div}(\rho u \otimes u)$, are rigorously re-examined in this step to ensure convergence under the weaker assumptions. In the second step, the result is extended to the whole domain by setting $\varphi=\psi_{\tau}\phi_{\delta}$ (with $\varphi=\psi_{\tau}$ for the torus $\mathbb{T}^{d}$) and letting $\tau, \delta \to 0$. The spatial cut-off function $\phi_{\delta}$ is constructed by the approaches used in \cite{chen2020energy} to control the regularity of the solution near the boundary. Particularly, the presence of a solid boundary complicates the dissipative mechanisms, as the behavior of the solution near the wall can differ significantly from its behavior in the interior. For further discussion on boundary effects in energy conservation contexts, we refer the reader to the related literature \cite{bardos2019onsager,bardos2019extension,drivas2018onsager}.

\emph{Weak-type commutator estimate.} Our main challenge is to justify the convergence of the commutator
$$
\partial_{t}(\rho u^{\varepsilon})_{\varepsilon}-\partial_{t}(\rho u_{\varepsilon}^{\varepsilon})
$$
under the Shinbrot-type condition on velocity. Compared with the previous work, these criteria entail a corresponding loss of regularity of $\partial_{t}\rho$, which is the key condition of the commutator estimate used in them. Since one usually employs the continuity equation:
$$
\partial_{t}\rho = - \text{div}(\rho u) = - (\rho \text{div} u + \nabla\rho\cdot u),
$$
to obtain integrability of $\partial_{t}\rho$, it may be difficult to derive the necessary $L_{t}^{p}L_{x}^{q}$ estimate of $\partial_{t}\rho$ under the assumption of the present paper. To overcome this obstacle, we prove Lemma \ref{commutator}, in which the convergence is established directly without establishing any $L^{p}$-type commutator estimate. Lemma \ref{commutator}, together with the viscous properties of the fluid, allows us to show that the error term vanishes even when $\partial_{t}\rho$ is only a distribution, which is ensured by the mass conservation.

\subsection{Main results}

Our results consist of the energy conservation criteria of two kinds of compressible Navier--Stokes systems: constant viscosities and degenerate viscosity. Here, we present our main results regarding the energy conservation for weak solutions to the equations with constant viscosities, whereas the remaining content, pertaining to degenerate viscosity cases, is outlined in the final section (Section \ref{sec:deg}). Throughout the present paper, let $\Omega$ be either the torus $\mathbb{T}^{N}$ ($N=2,3$) or an open, bounded domain in $\mathbb{R}^N$ with $C^1$ boundary $\partial \Omega$ with $u=0$ on $\partial\Omega$.

We first present the result for the Navier-Stokes equations \eqref{eq:comNS} with constant viscosity, where the presence of vacuum is allowed.

\begin{theorem}\label{theoremcomNS}
    Let $(\rho, u)$ be a weak solution of \eqref{eq:comNS} in the sense of Definition \ref{defcomNS}. Assume that the density satisfies
    \begin{equation}\label{con:thcomNSrhoover0}
         0 \leq \rho(t,x) \leq \bar{\rho} < \infty, \quad \text{and} \quad \nabla \sqrt{\rho} \in L^{\infty}\left(0, T ; L^{\frac{3}{2}}(\Omega)\right),
    \end{equation}
    and the initial velocity satisfies
    \begin{equation}\label{con:initial_1}
        u_0 \in L^{\frac{6N}{6-N}}(\Omega).
    \end{equation}
    If the velocity satisfies
    \begin{equation}\label{con:shinbrot_cond}
        u \in L^p\left(0, T ; L^q(\Omega)\right) \quad \text{with} \quad \begin{cases}
            \frac{1}{p}+\frac{3}{q} \leq 1, \quad\mbox{if} \quad 3\leq q < 4,\\
            \frac{2}{p}+\frac{2}{q} \leq 1, \quad\mbox{if} \quad 4\leq q <\infty,
        \end{cases}
    \end{equation}
    then the energy equality \eqref{eq:energy_comNS} holds for any $t \in[0, T]$.
\end{theorem}
\begin{remark}
    Our result improves the known regularity criteria in the literature and closes a notable gap. We relax the stronger Serrin-type integrability required in prior works (e.g., \eqref{con:Yu} and \eqref{con:Chen et al.}) via a novel weak-type temporal commutator estimate (Lemma \ref{commutator}), which handles the nonlinear terms without demanding strong regularity of $\partial_t\rho$. The extended range of admissible exponents is visualized in {\rm Figure 1}. Particularly, we provide that $\nabla\sqrt{\rho}\in L^{\infty}L^{2}$ instead of $L^{\infty}L^{2}$, which is also weaker than \cite{yu2017energy,chen2020energy}.
\end{remark}
\begin{remark}
     It is worth noting that, due to the non‑separability of the space $L^{2}(0,T;L^{\infty}(\Omega))$, our method does not currently apply to the endpoint $(p,q)=(2,\infty)$; this remains for future investigation. 
\end{remark}
\begin{remark}
    The approaches employed to validate this theorem are equally applicable to the compressible Navier-Stokes equations with degenerate viscosity, a fact we will demonstrate by proving Theorem \ref{theoremcomNSdeg} in Section 5.
\end{remark}

The remainder of this paper is organized as follows. Section 2 introduces the necessary preliminaries, including definitions of weak solutions and key technical lemmas. Section 3 focuses on establishing the local energy conservation law in the sense of distributions using a mollification argument. In Section 4, we extend this local result to the global energy balance on the domain $\Omega$, completing the proof of Theorem \ref{theoremcomNS}. Finally, we apply our methods on the compressible Navier-Stokes equations with degenerate viscosity in the last Section 5.

\section{Preliminaries}

First, we introduce some notations used in this paper. For $p \in[1, \infty]$, the notation $L^p(0, T ; X)$ stands for the set of measurable functions on the interval $(0, T)$ with values in $X$ and $\|f(t, \cdot)\|_X$ belonging to $L^p(0, T)$. The classical Sobolev space $W^{k, p}(\Omega)$ is equipped with the norm $\|f\|_{W^{k, p}(\Omega)}=\sum_{|\alpha|=0}^k\left\|D^\alpha f\right\|_{L^p(\Omega)}$. 

\begin{definition}[\cite{adams2003sobolev}]\label{def:neg_sobolev}
	Let $\Omega \subset \mathbb{R}^n$ be an open set, $m \ge 1$ be an integer, and $1 < p < \infty$. Let $p'$ be the conjugate exponent of $p$ satisfying $1/p + 1/p' = 1$. The space $W^{-m,p'}(\Omega)$ is defined as the dual space of the Sobolev space $W_0^{m,p}(\Omega)$.
    An element $T \in W^{-m,p'}(\Omega)$ is a bounded linear functional on $W_0^{m,p}(\Omega)$. The norm on this space, denoted by $\|\cdot\|_{-m,p'}$, is the standard operator norm:
	\begin{equation}\label{con:neg_norm}
		\|T\|_{-m,p'} = \sup_{u \in W_0^{m,p}(\Omega), \|u\|_{m,p} \le 1} |\langle T, u \rangle|,
	\end{equation}
	where $\langle \cdot, \cdot \rangle$ denotes the duality pairing.
\end{definition}

 Then, we provide the definitions of the weak solutions to \eqref{eq:comNS} and \eqref{eq:comNSdeg} with initial data \eqref{con:initial}. As for the global existence of these weak solutions, one can refer to \cite{bresch2003some,bresch2015two,feireisl2001existence,lions1996mathematical,vasseur2016existence} and the references therein.
\begin{definition}\label{defcomNS}
    For a given $T>0$, we call $(\rho, u)$ a weak solution on $[0, T]$ to the constant viscosity system \eqref{eq:comNS}\eqref{con:initial} if:
    \begin{itemize}
        \item The momentum equations \eqref{eq:comNS} hold in $\mathcal{D}^{\prime}((0, T) \times \Omega)$ with the regularity:
        \begin{equation}\label{con:defcomNS}
            \rho^\gamma, \rho|u|^2 \in L^{\infty}\left(0, T ; L^1(\Omega)\right), \quad u \in L^2\left(0, T ; H_0^1(\Omega)\right);
        \end{equation}
        
        \item The initial conditions \eqref{con:initial} hold in $\mathcal{D}^{\prime}(\Omega)$;

        \item $(\rho, u)$ is a renormalized solution of the continuity equation in the sense of \cite{diperna1989ordinary};
        
        \item The energy inequality holds for almost every $t \in [0,T]$:
        $$
        \begin{aligned}
            \int_{\Omega}\left(\frac{1}{2} \rho|u|^2+\frac{\rho^\gamma}{\gamma-1}\right) \mathrm{d} x & +\int_0^t \int_{\Omega}\left(\mu|\nabla u|^2+(\mu+\lambda)(\operatorname{div} u)^2\right) \mathrm{d} x \mathrm{d} s \\
            & \leq \int_{\Omega}\left(\frac{1}{2} \rho_0\left|u_0\right|^2+\frac{\rho_0^\gamma}{\gamma-1}\right) \mathrm{d} x.
        \end{aligned}
        $$
    \end{itemize}
\end{definition} 
Now, let us define the mollifiers. Let
\begin{equation}\label{mollifier}
f^{\varepsilon}(t, x)=\eta_{x,\varepsilon} * f,\quad f_{\varepsilon}(t, x)=\eta_{t,\varepsilon} * f \quad \text{for } t>\varepsilon,
\end{equation}
where $\eta_{x,\varepsilon}=\frac{1}{\varepsilon^{N}} \eta_{x}\left( \frac{x}{\varepsilon}\right),\eta_{t,\varepsilon}=\frac{1}{\varepsilon} \eta_{t}\left( \frac{t}{\varepsilon}\right)$. Here $\eta_{x}(x) \geq 0$ is a smooth even function compactly supported in the space ball of radius 1, and $\eta_{t}(t) \geq 0$ is a smooth even function compactly supported in the time ball of radius 1 with integral equal to 1.

	\begin{lemma} 
		 Let $p \in[1, \infty)$ and $f \in W_0^{1, p}(\Omega)$. There is a constant $C$ which depends on $p$ and $\Omega$, such that
		$$
		\left\|\frac{f(x)}{\operatorname{dist}(x, \partial \Omega)}\right\|_{L^p(\Omega)} \leq C\|f\|_{W_0^{1, p}(\Omega)}.
		$$
	\end{lemma}
	\begin{lemma}[\cite{mellet2007barotropic}]\label{aubinlions} Let $X \hookrightarrow B \hookrightarrow Y$ be three Banach spaces with compact embedding $X \hookrightarrow \hookrightarrow Y$. Further, let there exist $0<\theta<1$ and $M>0$ such that
	$$
	\|v\|_B \leq M\|v\|_X^{1-\theta}\|v\|_Y^\theta \text { for all } v \in X \cap Y \text {. }
	$$
	Denote for $T>0$,
	$$
	W(0, T):=W^{s_0, r_0}((0, T), X) \cap W^{s_1, r_1}((0, T), Y)
	$$
	with
	$$
	\begin{gathered}
		s_0, s_1 \in \mathbb{R} ; r_0, r_1 \in[1, \infty] \\
		s_\theta:=(1-\theta) s_0+\theta s_1, \frac{1}{r_\theta}:=\frac{1-\theta}{r_0}+\frac{\theta}{r_1}, s^*:=s_\theta-\frac{1}{r_\theta}.
	\end{gathered}
	$$
	Assume that $s_\theta>0$ and $F$ is a bounded set in $W(0, T)$. Then, we have 
	
	If $s_* \leq 0$, then $F$ is relatively compact in $L^p((0, T), B)$ for all $1 \leq p<p^*:=-\frac{1}{s^*}$. 
	
	If $s_*>0$, then $F$ is relatively compact in $C((0, T), B)$.
\end{lemma}

\begin{lemma}\label{commutator} 
Let $1 \leq \bar{p}, \bar{q}, p_1, q_1, p_2, q_2 \leq \infty$, with $\frac{1}{\bar{p}}+\frac{1}{p_1}+\frac{1}{p_2}=1$ and $\frac{1}{\bar{q}}+\frac{1}{q_1}+\frac{1}{q_2}=1$. Let $f\in L^{p_1}\left(0, T ; L^{q_1}(\Omega)\right), \partial_{t} f \in L^{p_1}\left(0, T ; W^{-1,q_1}(\Omega)\right), g \in L^{p_2}\left(0, T ; W^{1,q_2}(\Omega)\right), \varphi \in L^{\bar{p}}\left(0, T ; W_{0}^{1,\bar{q}}(\Omega)\right)$. Then, there holds
$$
\int_0^T \int_{\Omega}\varphi\left[\partial_t\left(f g\right)_{\varepsilon}-\partial_t(f g_{\varepsilon})\right]\mathrm{d} x \mathrm{d} t \rightarrow 0,
$$
as $\varepsilon \rightarrow 0$ if $p_2, q_2,\bar{p},\bar{q}<\infty$.
\end{lemma}
\begin{remark} 
     Lemma \ref{commutator} is of independent interest, as
    it is designed precisely for situations where the commutator acts on a test function
    and the derivative is defined as a distribution. Compared to the classical Lions's commutator estimate
    \begin{equation}\label{Lions}
    	\|\partial_t\left(f g\right)_{\varepsilon}-\partial_t(f g_{\varepsilon})\|_{\frac{p_{1}p_{2}}{p_{1}+p_{2}}}\leq\|\partial_{t}f\|_{p_{1}}\|g\|_{p_{2}},
    \end{equation}
    Lemma \ref{commutator} requires weaker regularity on $\partial_t f$ and obtains weaker convergence. Thereby, one can regard it as a distribution version of \eqref{Lions}.
\end{remark}
\begin{proof}
First, by a direct calculation, we know that
$$\int_0^T \int_{\Omega}\varphi\left[\partial_t\left(f g\right)_{\varepsilon}-\partial_t(f g_{\varepsilon})\right]\mathrm{d} x \mathrm{d} t=I-\int_0^T \int_{\Omega}\varphi\partial_tf g_{\varepsilon}\mathrm{d} x \mathrm{d} t,$$
where
$$
\begin{aligned}
	I=:&\int_0^T \int_{\Omega}\varphi\left[\partial_t\left(f g\right)_{\varepsilon}-f\partial_t g_{\varepsilon}\right]\mathrm{d} x \mathrm{d} t\\
    =&\int_{\varepsilon}^{T-\varepsilon}   \int_{t-\varepsilon}^{t+\varepsilon}\frac{1}{\varepsilon^{2}}\partial_{t} \eta_{t}\left(\frac{t-s}{\varepsilon}\right) \int_{\Omega} \left[f(s,x)-f(t,x)\right] g(s, x)\varphi(t,x)\mathrm{d}x \mathrm{d} s \mathrm{d}t\\
	=&\int_{\varepsilon}^{T-\varepsilon}   \int_{t-\varepsilon}^{t+\varepsilon}\frac{1}{\varepsilon^{2}}\partial_{t} \eta_{t}\left(\frac{t-s}{\varepsilon}\right) \int_{t}^{s} \left\langle \partial_{t}f(\tau, x)  , g(s, x)\varphi(t,x)\right\rangle\mathrm{d} \tau\mathrm{d} s \mathrm{d}t.
\end{aligned}
$$
As
$$
\begin{aligned}
	&\left\langle \partial_{t}f(\tau, x)  , g(s, x)\varphi(t,x)\right\rangle\\
    \leq&\| \partial_{t}f(\tau)\|_{W^{-1,q_{1}}(\Omega)}\left(\|\nabla g(s)\|_{L^{q_{2}}(\Omega)}\|\varphi(t)\|_{L^{\bar{q}}(\Omega)}+\| g(s)\|_{L^{q_{2}}(\Omega)}\|\nabla\varphi(t)\|_{L^{\bar{q}}(\Omega))}\right)
\end{aligned}
$$
then
$$
\begin{aligned}
	\int_0^T \int_{\Omega}\varphi\partial_tf g_{\varepsilon}\mathrm{d} x \mathrm{d} t
	=&\int_0^T \left\langle \partial_{t}f(t, x)  , g(t, x)\varphi(t,x)\right\rangle \mathrm{d} t\\
	\leq& \| \partial_{t}f\|_{L^{p_{1}}(0,T;W^{-1,q_{1}}(\Omega))}\|g\|_{L^{p_{2}}(0,T;W^{1,q_{2}}(\Omega))}\|\varphi\|_{L^{\bar{p}}(0,T;W^{1,\bar{q}}(\Omega))}.
\end{aligned}$$
Meanwhile we have
$$
\begin{aligned}
	I\leq&\int_{\varepsilon}^{T-\varepsilon}   \int_{t-\varepsilon}^{t+\varepsilon}\frac{1}{\varepsilon^{2}}\partial_{t} \eta_{t}\left(\frac{t-s}{\varepsilon}\right) \int_{[s,t]} \| \partial_{t}f(\tau)\|_{W^{-1,q_{1}}(\Omega)}\|\nabla g(s)\|_{L^{q_{2}}(\Omega)}\|\varphi(t)\|_{L^{\bar{q}}(\Omega)}\mathrm{d} \tau\mathrm{d} s \mathrm{d}t\\
	&+\int_{\varepsilon}^{T-\varepsilon}   \int_{t-\varepsilon}^{t+\varepsilon}\frac{1}{\varepsilon^{2}}\partial_{t} \eta_{t}\left(\frac{t-s}{\varepsilon}\right) \int_{[s,t]} \| \partial_{t}f(\tau)\|_{W^{-1,q_{1}}(\Omega)}\| g(s)\|_{L^{q_{2}}(\Omega)}\|\nabla\varphi(t)\|_{L^{\bar{q}}(\Omega)}\mathrm{d} \tau\mathrm{d} s \mathrm{d}t\\
    \leq&\int_{\varepsilon}^{T-\varepsilon}   \int_{t-\varepsilon}^{t+\varepsilon}\frac{1}{\varepsilon^{2}}\partial_{t} \eta_{t}\left(\frac{t-s}{\varepsilon}\right)  \left(\| \partial_{t}f\|_{W^{-1,q_{1}}(\Omega)}\ast\boldsymbol{1}_{[-\varepsilon,\varepsilon]}\right)(t)\|\nabla g(s)\|_{L^{q_{2}}(\Omega)}\|\varphi(t)\|_{L^{\bar{q}}(\Omega)}\mathrm{d} s \mathrm{d}t\\
	&+\int_{\varepsilon}^{T-\varepsilon}   \int_{t-\varepsilon}^{t+\varepsilon}\frac{1}{\varepsilon^{2}}\partial_{t} \eta_{t}\left(\frac{t-s}{\varepsilon}\right) \left(\| \partial_{t}f\|_{W^{-1,q_{1}}(\Omega)}\ast\boldsymbol{1}_{[-\varepsilon,\varepsilon]}\right)(t)\| g(s)\|_{L^{q_{2}}(\Omega)}\|\nabla\varphi(t)\|_{L^{\bar{q}}(\Omega)}\mathrm{d} s \mathrm{d}t\\
	\leq&\| \partial_{t}f\|_{L^{p_{1}}(0,T;W^{-1,q_{1}}(\Omega))}\|g\|_{L^{p_{2}}(0,T;W^{1,q_{2}}(\Omega))}\|\varphi\|_{L^{\bar{p}}(0,T;W^{1,\bar{q}}(\Omega))},
\end{aligned}
$$
hence we obtain that
\[
\begin{split}
	&\int_0^T \int_{\Omega}\varphi\left[\partial_t\left(f g\right)_{\varepsilon}-\partial_t(f g_{\varepsilon})\right]\mathrm{d} x \mathrm{d} t\\
	 \leq&\| \partial_{t}f\|_{L^{p_{1}}(0,T;W^{-1,q_{1}}(\Omega))}\|g\|_{L^{p_{2}}(0,T;W^{1,q_{2}}(\Omega))}\|\varphi\|_{L^{\bar{p}}(0,T;W^{1,\bar{q}}(\Omega))}. 
\end{split}\]
$$$$
Furthermore, if $1 \leq p_2, q_2<\infty$, let $\left\{g_n\right\} \in C_0^{\infty}((0,T)\times\Omega)$ with $g_n \rightarrow g$ strongly in $ L^{p_2}(W^{1,q_2})$. Thus, by the density arguments and properties of the standard mollifiers, we arrive at
$$
\begin{aligned}
&\int_0^T \int_{\Omega}\varphi\left[\partial_t\left(f g\right)_{\varepsilon}-\partial_t(f g_{\varepsilon})\right]\mathrm{d} x \mathrm{d} t\\
=&\int_0^T \int_{\Omega}\varphi\left[\partial_t\left(f (g-g_n)\right)_{\varepsilon}-\partial_t(f (g-g_n)_{\varepsilon})\right]\mathrm{d} x \mathrm{d} t
+\int_0^T \int_{\Omega}\varphi\left[\partial_t\left(f g_n\right)_{\varepsilon}-\partial_t(f g_{n\varepsilon})\right]\mathrm{d} x \mathrm{d} t\\
= &\int_0^T \int_{\Omega}\varphi\left[\partial_t\left(f (g-g_n)\right)_{\varepsilon}-\partial_t(f (g-g_n)_{\varepsilon})\right]\mathrm{d} x \mathrm{d} t
+\int_0^T \int_{\Omega}\varphi\left[\left(f\partial_t g_n\right)_{\varepsilon}-(f\partial_t g_{n\varepsilon})\right]\mathrm{d} x \mathrm{d} t \\
&+\int_0^T \int_{\Omega}\varphi\left[\left(\partial_tf g_n\right)_{\varepsilon}-(\partial_tf g_{n\varepsilon})\right]\mathrm{d} x \mathrm{d} t,
\end{aligned}$$
and
$$
\begin{aligned}
I=&\int_0^T \int_{\Omega}\varphi\left[\partial_t\left(f (g-g_n)\right)_{\varepsilon}-\partial_t(f (g-g_n)_{\varepsilon})\right]\mathrm{d} x \mathrm{d} t
+\int_0^T \int_{\Omega}\varphi\left[\left(f\partial_t g_n\right)_{\varepsilon}-(f\partial_t g_{n\varepsilon})\right]\mathrm{d} x \mathrm{d} t\\
&+\int_0^T \left\langle \partial_{t}f , g_n\varphi_{\varepsilon}-g_n\varphi\right\rangle  +\left\langle \partial_tf, g_n\varphi- g_{n\varepsilon}\varphi\right\rangle \mathrm{d} x \mathrm{d} t\\
\leq &\| \partial_{t}f\|_{L^{p_{1}}(0,T;W^{-1,q_{1}}(\Omega))}\|g-g_n\|_{L^{p_{2}}(0,T;W^{1,q_{2}}(\Omega))}\|\varphi\|_{L^{\bar{p}}(0,T;W^{1,\bar{q}}(\Omega))}\\
&+\| f\|_{L^{p_{1}}(0,T;L^{q_{1}}(\Omega))}\|\partial_{t}g_{n}\|_{L^{p_{2}}(0,T;L^{q_{2}}(\Omega))}\|\varphi_{\varepsilon}-\varphi\|_{L^{\bar{p}}(0,T;L^{\bar{q}}(\Omega))}\\
&+\| f\|_{L^{p_{1}}(0,T;L^{q_{1}}(\Omega))}\|\partial_{t}g_{n}-\partial_{t}g_{n\varepsilon}\|_{L^{p_{2}}(0,T;L^{q_{2}}(\Omega))}\|\varphi\|_{L^{\bar{p}}(0,T;L^{\bar{q}}(\Omega))}\\
&+\| \partial_{t}f\|_{L^{p_{1}}(0,T;W^{-1,q_{1}}(\Omega))}\|g\|_{L^{p_{2}}(0,T;W^{1,q_{2}}(\Omega))}\|\varphi_{\varepsilon}-\varphi\|_{L^{\bar{p}}(0,T;W^{1,\bar{q}}(\Omega))}\\
&+\| \partial_{t}f\|_{L^{p_{1}}(0,T;W^{-1,q_{1}}(\Omega))}\|g_n-g_{n\varepsilon}\|_{L^{p_{2}}(0,T;W^{1,q_{2}}(\Omega))}\|\varphi\|_{L^{\bar{p}}(0,T;W^{1,\bar{q}}(\Omega))}
\rightarrow 0,
\end{aligned}$$
as $\varepsilon \rightarrow 0$.

\end{proof}
\section{Local energy equality of the compressible Navier-Stokes equations with constant viscosity}
\begin{theorem}\label{localenergy}

Let $(\rho, u)$ be a weak solution to the compressible Navier-Stokes equations \eqref{eq:comNS}. If assumptions \eqref{condi1-th2} and \eqref{condi2-th2} hold,
then the energy conservation holds in the sense of distributions, i.e.
\begin{equation}\label{equ-local energy}
	\begin{split}
		&\frac{1}{2}\int_0^T \int_{\Omega}\partial_{t}\varphi \rho |u |^{2}\mathrm{d} x \mathrm{d} t+\int_0^T \int_{\Omega}\varphi \left[\rho^{\gamma}{\rm div}u+\mu|\nabla u|^{2}+(\mu+\lambda)({\rm div}u)^{2}\right] \mathrm{d} x \mathrm{d} t\\
		&\quad+\int_0^T \int_{\Omega}\nabla\varphi \cdot\left[\frac{1}{2}(\rho u) |u |^{2}+\rho^{\gamma}u +\mu u\nabla u+(\mu+\lambda)u {\rm div}u\right]\mathrm{d} x \mathrm{d} t=0.
	\end{split}
	\end{equation}
where $\varphi \in C_c^\infty((0, T) \times \Omega)$.
\end{theorem}
\begin{proof}

We only need to treat the cases that $2<p\leq4$. In fact, we have the following embedding
$$
L^{p}(0,T;L^{q}(\Omega))\cap L^{2}(0,T;H^{1}(\Omega))\hookrightarrow L^{4}(0,T;L^{4}(\Omega)),\quad\mbox{where}\quad \frac{1}{p}+\frac{3}{q}=1\quad\mbox{and}\quad p>4.
$$
Multiplying \eqref{eq:comNS} by test function $(\varphi u_{\varepsilon}^{\varepsilon})_{\varepsilon}^{\varepsilon}$, then integrating over $(0, T) \times \Omega$, we infer that
\begin{equation}\label{eq:thm_1}
	\int_0^T \int_{\Omega} \varphi u_{\varepsilon}^{\varepsilon}\left[\partial_t(\rho u)_{\varepsilon}^{\varepsilon}+\operatorname{div}(\rho u \otimes u)_{\varepsilon}^{\varepsilon}+\nabla P(\rho)_{\varepsilon}^{\varepsilon}-\mu \Delta u_{\varepsilon}^{\varepsilon}-(\mu+\lambda) \nabla(\operatorname{div} u)_{\varepsilon}^{\varepsilon}\right]=0, 
	\end{equation}
where we used the fact $\eta(-t,-x)=\eta(t, x)$. In the rest of this step, we will reformulate each term of the last equation.

We first show that
$$
\begin{aligned}
	& \lim _{\varepsilon \rightarrow 0} \int_0^T \int_{\Omega} \varphi u_{\varepsilon}^{\varepsilon} \partial_t(\rho u)_{\varepsilon}^{\varepsilon}\mathrm{d} x \mathrm{d} t
	 +\lim _{\varepsilon \rightarrow 0} \int_0^T \int_{\Omega} \varphi u_{\varepsilon}^{\varepsilon}\operatorname{div}(\rho u \otimes u)_{\varepsilon}^{\varepsilon} \mathrm{d} x \mathrm{d} t\\
 =&	-\frac{1}{2} \int_0^T \int_{\Omega} \varphi_{t} \rho|u|^2\mathrm{d} x \mathrm{d} t-\frac{1}{2} \int_0^T \int_{\Omega}  \rho u \cdot \nabla \varphi|u|^2 \mathrm{d} x \mathrm{d} t.
\end{aligned}
$$
Firstly, for the time derivative term, we have 
$$
\begin{aligned}
	& \int_0^T \int_{\Omega} \varphi u_{\varepsilon}^{\varepsilon} \partial_t(\rho u)_{\varepsilon}^{\varepsilon}\mathrm{d} x \mathrm{d} t \\
	=&\int_0^T \int_{\Omega} \varphi u_{\varepsilon}^{\varepsilon} \left[\partial_t(\rho u)_{\varepsilon}^{\varepsilon}-\partial_t(\rho u_{\varepsilon}^{\varepsilon})\right] \mathrm{d} x \mathrm{d} t+\int_0^T \int_{\Omega} \varphi u_{\varepsilon}^{\varepsilon} \partial_t(\rho u_{\varepsilon}^{\varepsilon})\mathrm{d} x \mathrm{d} t  \\
	 =:& I_1+\int_0^T \int_{\Omega} \varphi u_{\varepsilon}^{\varepsilon} \partial_t\left(\rho u_{\varepsilon}^{\varepsilon}\right)\mathrm{d} x \mathrm{d} t.
\end{aligned}
$$
A straightforward computation gives
   $$
		\begin{aligned}
	I_1
	=&\int_0^T \int_{\Omega} \varphi u_{\varepsilon}^{\varepsilon} \left[\partial_t(\rho u)_{\varepsilon}^{\varepsilon}-\partial_t(\rho u^{\varepsilon})_{\varepsilon}\right] \mathrm{d} x \mathrm{d} t+\int_0^T \int_{\Omega} \varphi u_{\varepsilon}^{\varepsilon} \left[\partial_t(\rho u^{\varepsilon})_{\varepsilon}-\partial_t(\rho u_{\varepsilon}^{\varepsilon})\right] \mathrm{d} x \mathrm{d} t\\
	=&:M_1+M_2 .
		\end{aligned}
	$$
Let us show $\lim _{\varepsilon \rightarrow 0} I_{1}=0$. \\

        For any weak solution $(\rho, u)$  in the sense of Definition \ref{defcomNS}, with additional condition \eqref{condi2-th2}, then there exists some $\alpha \in(0,1)$ such that  $\partial_{t} \rho \in L^{r}(0, T ; W^{-1,s}(\Omega))$, for any
        $$\frac{1}{r}=\frac{\alpha}{p}\text{ and } \frac{1}{s}=\frac{\alpha}{q}+\frac{1-\alpha}{2}.$$

    Note that
$ \rho u \in L^{\infty}(0, T ; L^2(\Omega)) \text{ and } u \in L^p\left(0, T ; L^q(\Omega)\right)$, 
making use of $L^p-L^q$ interpolation inequality,we deduce that there exists some $\alpha \in(0,1)$ such that
$$\rho u \in L^{r}(0, T ; L^{s}(\Omega)) \text{ with }\frac{1}{r}=\frac{\alpha}{p},\quad \frac{1}{s}=\frac{\alpha}{q}+\frac{1-\alpha}{2}.$$
By the mass equation $\eqref{eq:comNS}_{2}$, one has 
$$\partial_{t} \rho \in L^{r}(0, T ; W^{-1,s}(\Omega)) .$$

For $M_1$, we have
    $$
    \begin{aligned}
    &\int_0^T \int_{\Omega} \varphi u_{\varepsilon}^{\varepsilon} \left[\partial_t(\rho u)_{\varepsilon}^{\varepsilon}-\partial_t(\rho u^{\varepsilon})_{\varepsilon}\right] \mathrm{d} x \mathrm{d} t\\
    =&\int_0^T \int_{\Omega} \varphi u_{\varepsilon}^{\varepsilon} \left[\left((\rho u)\ast\eta_{x,\varepsilon}-\rho\cdot u\ast\eta_{x,\varepsilon}\right)\ast \eta_{t,\varepsilon}'\right] \mathrm{d} x \mathrm{d} t\\
          =&\int_0^T \int_{\Omega} \varphi u_{\varepsilon}^{\varepsilon} \left[\int_{\boldsymbol{B}(0, \varepsilon)} (\rho(t,x-y)-\rho(t,x))u(t,x-y)\eta_{x,\varepsilon}(y)\mathrm{d} y \ast \eta_{t,\varepsilon}'\right] \mathrm{d} x \mathrm{d} t\\
        =&\int_0^T \int_{\Omega} \varphi u_{\varepsilon}^{\varepsilon} \left[\int_{-\varepsilon}^{\varepsilon} \int_{\boldsymbol{B}(0, \varepsilon)} \int_{0}^{1} \nabla \rho(t-s,x-\tau y)\cdot y \mathrm{d}\tau u(t-s,x-y)\eta_{x,\varepsilon}(y)\eta_{t,\varepsilon}'(s)\mathrm{d} y \mathrm{d} s \right] \mathrm{d} x \mathrm{d} t\\
        =&\int_0^T \int_{\Omega} \varphi u_{\varepsilon}^{\varepsilon} \left[\int_{-\varepsilon}^{\varepsilon} \int_{\boldsymbol{B}(0, \varepsilon)} \int_{0}^{1} \nabla \rho(t-s,x-\tau y)\cdot y  u(t-s,x-y)\eta_{x,\varepsilon}(y)\frac{1}{\varepsilon^{2}}\partial_{t}\eta_{t}(\frac{s}{\varepsilon})\mathrm{d}\tau\mathrm{d} y \mathrm{d} s \right] \mathrm{d} x \mathrm{d} t\\
        =&  \int_{-\varepsilon}^{\varepsilon} \int_{\boldsymbol{B}(0, \varepsilon)} \int_{0}^{1}\int_0^T \int_{\Omega}\varphi u_{\varepsilon}^{\varepsilon} \nabla \rho(t-s,x-\tau y)\cdot y  u(t-s,x-y)\mathrm{d} x \mathrm{d} t\eta_{x,\varepsilon}(y)\frac{1}{\varepsilon^{2}}\partial_{t}\eta_{t}(\frac{s}{\varepsilon})\mathrm{d}\tau\mathrm{d} y \mathrm{d} s\\
    	\leq&  \int_{-\varepsilon}^{\varepsilon} \int_{\boldsymbol{B}(0, \varepsilon)} \int_{0}^{1}\int_0^T \int_{\Omega}\varphi u_{\varepsilon}^{\varepsilon} \nabla \rho(t-s,x-\tau y)  u(t-s,x-y)\mathrm{d} x \mathrm{d} t\eta_{x,\varepsilon}(y)\frac{1}{\varepsilon}\partial_{t}\eta_{t}(\frac{s}{\varepsilon})\mathrm{d}\tau\mathrm{d} y \mathrm{d} s  \\
        \leq &\int_{-\varepsilon}^{\varepsilon} \int_{\boldsymbol{B}(0, \varepsilon)}\int_{0}^{1} \|u_{\varepsilon}^{\varepsilon}\|_{L^{2}(L^{6})}\| \nabla \rho(t-s,x-\tau y)\|_{L^{\infty}(L^{\frac{3}{2}})}  \|u(t-s,x-y)\|_{L^{2}(L^{6})}\eta_{x,\varepsilon}(y)\frac{1}{\varepsilon} \partial_{t}\eta_{t}(\frac{s}{\varepsilon})\mathrm{d}\tau\mathrm{d} y \mathrm{d} s,
    \end{aligned}
    $$
    hence
    $$
    \begin{aligned}
        M_{1}\leq & \|u\|_{L^{2}(L^{6})}\| \nabla \rho\|_{L^{\infty}(L^{\frac{3}{2}})}  \|u\|_{L^{2}(L^{6})} \int_{-\varepsilon}^{\varepsilon} \int_{\boldsymbol{B}(0, \varepsilon)}\int_{0}^{1} \eta_{x,\varepsilon}(y)\frac{1}{\varepsilon} \partial_{t}\eta_{t}(\frac{s}{\varepsilon})\mathrm{d}\tau\mathrm{d} y \mathrm{d} s\\
        \leq &C\|u\|_{L^{2}(L^{6})}\| \nabla \rho\|_{L^{\infty}(L^{\frac{3}{2}})}  \|u\|_{L^{2}(L^{6})}.
    \end{aligned}
	$$
    Furthermore,let $\left\{u_n\right\} \in C_0^{\infty}((0,T)\times\Omega)$ with $u_n \rightarrow u$ strongly in $L^{p}(L^{q}) \cap L^{2}(W^{1,2})$. Thus, by the density arguments and properties of the standard mollifiers, we arrive at
    $$
    \begin{aligned}
        &\int_0^T \int_{\Omega} \varphi u_{\varepsilon}^{\varepsilon} \left[\partial_t(\rho u)_{\varepsilon}^{\varepsilon}-\partial_t(\rho u^{\varepsilon})_{\varepsilon}\right] \mathrm{d} x \mathrm{d} t\\
        =&\int_0^T \int_{\Omega} \varphi u_{\varepsilon}^{\varepsilon} \left[\partial_t(\rho (u-u_{n}))_{\varepsilon}^{\varepsilon}-\partial_t(\rho (u-u_{n})^{\varepsilon})_{\varepsilon}\right] \mathrm{d} x \mathrm{d} t+\int_0^T \int_{\Omega} \varphi u_{\varepsilon}^{\varepsilon} \left[\partial_t(\rho u_{n})_{\varepsilon}^{\varepsilon}-\partial_t(\rho u_{n}^{\varepsilon})_{\varepsilon}\right] \mathrm{d} x \mathrm{d} t\\
        =&\int_0^T \int_{\Omega} \varphi u_{\varepsilon}^{\varepsilon} \left[\partial_t(\rho (u-u_{n}))_{\varepsilon}^{\varepsilon}-\partial_t(\rho (u-u_{n})^{\varepsilon})_{\varepsilon}\right] \mathrm{d} x \mathrm{d} t+\int_0^T \int_{\Omega} \left(\varphi u_{\varepsilon}^{\varepsilon}\right)_{\varepsilon} \left[\partial_t(\rho u_{n})^{\varepsilon}-\partial_t(\rho u_{n}^{\varepsilon})\right] \mathrm{d} x \mathrm{d} t\\
        =&\int_0^T \int_{\Omega} \varphi u_{\varepsilon}^{\varepsilon} \left[\partial_t(\rho (u-u_{n}))_{\varepsilon}^{\varepsilon}-\partial_t(\rho (u-u_{n})^{\varepsilon})_{\varepsilon}\right] \mathrm{d} x \mathrm{d} t+\int_0^T \int_{\Omega} \left(\varphi u_{\varepsilon}^{\varepsilon}\right)_{\varepsilon} \left[(\rho \partial_tu_{n})^{\varepsilon}-(\rho \partial_tu_{n}^{\varepsilon})\right] \mathrm{d} x \mathrm{d} t\\
        &+\int_0^T \int_{\Omega} \left(\varphi u_{\varepsilon}^{\varepsilon}\right)_{\varepsilon} \left[(\partial_t\rho u_{n})^{\varepsilon}-(\partial_t\rho u_{n}^{\varepsilon})\right] \mathrm{d} x \mathrm{d} t.
    \end{aligned}
    $$
    Then,
    $$
    \begin{aligned}
    M_{1}
       =&\int_0^T \int_{\Omega} \varphi u_{\varepsilon}^{\varepsilon} \left[\partial_t(\rho (u-u_{n}))_{\varepsilon}^{\varepsilon}-\partial_t(\rho (u-u_{n})^{\varepsilon})_{\varepsilon}\right] \mathrm{d} x \mathrm{d} t+\int_0^T \int_{\Omega} \left(\varphi u_{\varepsilon}^{\varepsilon}\right)_{\varepsilon} \left[(\rho \partial_tu_{n})^{\varepsilon}-(\rho \partial_tu_{n}^{\varepsilon})\right] \mathrm{d} x \mathrm{d} t\\
        &+\int_0^T \langle \partial_t\rho,(\varphi u_{\varepsilon}^{\varepsilon})^{\varepsilon}_{\varepsilon}u_{n} -(\varphi u_{\varepsilon}^{\varepsilon})_{\varepsilon}u_{n}\rangle  +\langle \partial_t\rho,(\varphi u_{\varepsilon}^{\varepsilon})_{\varepsilon}u_{n}^{\varepsilon}-(\varphi u_{\varepsilon}^{\varepsilon})_{\varepsilon}u_{n}\rangle \mathrm{d} x \mathrm{d} t\\  
        \leq& C\|u\|_{L^{2}(L^{6})}\| \nabla \rho\|_{L^{\infty}(L^{\frac{3}{2}})}  \|u-u_{n}\|_{L^{2}(L^{6})}+\|u\|_{L^{2}(L^{6})}\|(\rho \partial_tu_{n})^{\varepsilon}-(\rho \partial_tu_{n}^{\varepsilon})\|_{L^{2}(L^{\frac{6}{5}})}\\
        &+\| \partial_{t}\rho \|_{L^{r}(W^{-1,s})}\| u_{n}\|_{L^{2}(W^{1,2})}\|(\varphi u_{\varepsilon}^{\varepsilon})_{\varepsilon}^{\varepsilon}-(\varphi u_{\varepsilon}^{\varepsilon})_{\varepsilon}\|_{L^{p}(L^{q})}\\
        &+\| \partial_{t}\rho \|_{L^{r}(W^{-1,s})}\|u_{n}\|_{L^{p}(L^{q})}\| (\varphi u_{\varepsilon}^{\varepsilon})_{\varepsilon}^{\varepsilon}- (\varphi u_{\varepsilon}^{\varepsilon})_{\varepsilon}\|_{L^{2}(W^{1,2})}\\
&+\| \partial_{t}\rho \|_{L^{r}(W^{-1,s})}\| u_n- u_{n}^{\varepsilon}\|_{L^{2}(W^{1,2})}\|(\varphi u_{\varepsilon}^{\varepsilon})_{\varepsilon}\|_{L^{p}(L^{q})}\\
&+\| \partial_{t}\rho \|_{L^{r}(W^{-1,s})}\|u_n-u_{n}^{\varepsilon}\|_{L^{p}(L^{q})} \| (\varphi u_{\varepsilon}^{\varepsilon})_{\varepsilon}\|_{L^{2}(W^{1,2})}\\
&\rightarrow 0, \text{ as } \varepsilon \rightarrow 0,
    \end{aligned}$$
    provided that $\frac{1}{p}+\frac{1}{q} \leq \frac{1}{2},2<p\leq 4 $.

    For $M_2$, according to Lemma \ref{commutator} by  considering $f=\rho,g=u^{\varepsilon},\varphi=\varphi u_{\varepsilon}^{\varepsilon}$, we end up with
    \begin{equation}
    |M_2| = \left|\int_0^T \int_{\Omega} \varphi u_{\varepsilon}^{\varepsilon} \left[\partial_t(\rho u^{\varepsilon})_{\varepsilon}-\partial_t(\rho u_{\varepsilon}^{\varepsilon})\right] \mathrm{d} x \mathrm{d} t\right| \rightarrow 0,\text{ as }\varepsilon \rightarrow 0.
\end{equation}

Secondly, the convection term can be treated as
$$
\begin{aligned}
	\int_{0}^{T} \int_{\Omega} \varphi u_{\varepsilon}^{\varepsilon}\operatorname{div}(\rho u \otimes u)_{\varepsilon}^{\varepsilon}\mathrm{d}x\mathrm{d}t  = &\int_{0}^{T} \int_{\Omega} \varphi u_{\varepsilon}^{\varepsilon} \left( \operatorname{div}(\rho u \otimes u)_{\varepsilon}^{\varepsilon} - \operatorname{div}(\rho u \otimes u_{\varepsilon}^{\varepsilon}) \right) \mathrm{d}x\mathrm{d}t\\
	& + \int_{0}^{T} \int_{\Omega} \varphi u_{\varepsilon}^{\varepsilon}\operatorname{div}(\rho u \otimes u_{\varepsilon}^{\varepsilon})\mathrm{d}x\mathrm{d}t \\
	 =&: I_2 + \int_{0}^{T} \int_{\Omega} \varphi u_{\varepsilon}^{\varepsilon}\operatorname{div}(\rho u \otimes u_{\varepsilon}^{\varepsilon})\mathrm{d}x\mathrm{d}t.
\end{aligned}
$$
We claim that $\lim_{\varepsilon \to 0} I_2 = 0$. In fact, making use of the integration by parts, one has
$$
I_2 =-\int_{0}^{T} \int_{\Omega}  \nabla(\varphi u_{\varepsilon}^{\varepsilon}) \cdot \left((\rho u \otimes u)_{\varepsilon}^{\varepsilon} - (\rho u \otimes u_{\varepsilon}^{\varepsilon}) \right)\mathrm{d}x\mathrm{d}t.
$$
Due to the properties of the mollifier and H\"{o}lder inequality,
$$
\begin{aligned}
	|I_{2}| &\leq C(\|\nabla u\|_{L^{2}(L^{2})}+\| u\|_{L^{2}(L^{6})})\|(\rho u \otimes u)_{\varepsilon}^{\varepsilon} -\rho u \otimes u\|_{L^{2}(L^{2})} \\
	&\quad +C(\|\nabla u\|_{L^{2}(L^{2})}+\| u\|_{L^{2}(L^{6})})\|u\|^{\frac{p}{2}-1}_{L^{p}(L^{q})} \|\rho u\|^{2-\frac{p}{2}}_{L^{\infty}(L^{2})}\|u-u_{\varepsilon}^{\varepsilon}\|_{L^{p}(L^{q})}\\
	&\rightarrow 0, \text{ as }\varepsilon \rightarrow 0,
\end{aligned}
$$
provided that $\frac{1}{p}+\frac{1}{q} \leq \frac{1}{2},2<p\leq 4 $.

Finally, by the mass equation, a simple computation gives
\begin{equation}\label{equ-}
	\begin{aligned}
		 &\quad \int_0^T \int_{\Omega} \varphi u_{\varepsilon}^{\varepsilon} \partial_t\left(\rho u_{\varepsilon}^{\varepsilon}\right)\mathrm{d}x\mathrm{d}t+\int_{0}^{T} \int_{\Omega} \varphi u_{\varepsilon}^{\varepsilon}\operatorname{div}(\rho u \otimes u_{\varepsilon}^{\varepsilon})\mathrm{d}x\mathrm{d}t\\
		 &=  \int_0^T \int \varphi \rho \partial_t \frac{\left|u_{\varepsilon}^{\varepsilon}\right|^2}{2} \mathrm{d}x\mathrm{d}t
		 +\int_0^T \int \varphi \rho_t\left|u_{\varepsilon}^{\varepsilon}\right|^2\mathrm{d}x\mathrm{d}t + \frac{1}{2}\int_0^T \int \varphi\operatorname{div}(\rho u\left|u_{\varepsilon}^{\varepsilon}\right|^2)\mathrm{d}x\mathrm{d}t\\
		 &=-\frac{1}{2}\int_0^T \int \varphi_{t} \rho\left|u_{\varepsilon}^{\varepsilon}\right|^2\mathrm{d}x\mathrm{d}t
		 -\frac{1}{2}\int_0^T \int  \nabla \varphi\rho u\left|u_{\varepsilon}^{\varepsilon}\right|^2\mathrm{d}x\mathrm{d}t\\
		 &\rightarrow -\frac{1}{2}\int_0^T \int \varphi_{t} \rho\left|u\right|^2\mathrm{d}x\mathrm{d}t
		 -\frac{1}{2}\int_0^T \int  \nabla \varphi\rho u\left|u\right|^2\mathrm{d}x\mathrm{d}t, \text{ as }\varepsilon \rightarrow 0.
	\end{aligned}
\end{equation}
Combining the limits for $I_1$ and $I_2$ with the final calculation establishes the claim for the inertial terms.

Next, we show that
$$
	\lim_{\varepsilon\to 0}\int_{0}^{T}\int_{\Omega}\varphi u_{\varepsilon}^{\varepsilon}\nabla(\rho^{\gamma})_{\varepsilon}^{\varepsilon}\mathrm{d}x\mathrm{d}t = \int_{0}^{T}\int_{\Omega}\varphi u\cdot\nabla\rho^{\gamma} \mathrm{d}x\mathrm{d}t.
$$
We write
$$
	\begin{aligned}
		\int_{0}^{T}\int_{\Omega}\varphi u_{\varepsilon}^{\varepsilon}\nabla(\rho^{\gamma})_{\varepsilon}^{\varepsilon}\mathrm{d}x\mathrm{d}t &=\int_{0}^{T}\int_{\Omega}\varphi u_{\varepsilon}^{\varepsilon}\left(\nabla(\rho^{\gamma})_{\varepsilon}^{\varepsilon}-\nabla\rho^{\gamma}\right)\mathrm{d}x\mathrm{d}t+\int_{0}^{T}\int_{\Omega}\varphi u_{\varepsilon}^{\varepsilon}\cdot\nabla\rho^{\gamma}\mathrm{d}x\mathrm{d}t\\
		&=:I_{3}+\int_{0}^{T}\int_{\Omega}\varphi u_{\varepsilon}^{\varepsilon}\cdot\nabla\rho^{\gamma}\mathrm{d}x\mathrm{d}t.
	\end{aligned}
$$
Given that $\nabla\rho^{\gamma}\in L^{\infty}(L^{\frac{3}{2}})$, we know $\lim_{\varepsilon \to 0} \| \nabla((\rho^\gamma)_{\varepsilon}^{\varepsilon} - \rho^\gamma) \|_{L^2(L^{\frac{3}{2}})} = 0$. Hence,
$$
\lim_{\varepsilon \to 0} |I_3| \leq C \lim_{\varepsilon \to 0} \| u \|_{L^{2}(L^{6})} \|\nabla((\rho^\gamma)_{\varepsilon}^{\varepsilon} - \rho^\gamma) \|_{L^{2}(L^{\frac{3}{2}})} = 0.
$$

For the diffusion terms, we show that
$$
\begin{aligned}
	& \lim_{\varepsilon \to 0} \int_0^T \int_\Omega \varphi  u_{\varepsilon}^{\varepsilon} (\mu \Delta u_{\varepsilon}^{\varepsilon} + (\mu + \lambda) \nabla \mathrm{div} u_{\varepsilon}^{\varepsilon}) \mathrm{d}x\mathrm{d}t \\
	& = -\int_0^T \int_\Omega \varphi  (\mu |\nabla u|^2 + (\mu + \lambda)(\mathrm{div} u)^2) \mathrm{d}x\mathrm{d}t - \int_0^T \int_\Omega  \nabla \varphi \cdot (\mu u \nabla u + (\mu + \lambda) u \mathrm{div} u) \mathrm{d}x\mathrm{d}t.
\end{aligned}
$$
Making use of the integration by parts and the fact that $\nabla u \in L^{2}(L^{2})$, one has
$$
\begin{aligned}
	&-\int_0^T \int_\Omega \varphi u_{\varepsilon}^{\varepsilon}(\mu \Delta u_{\varepsilon}^{\varepsilon} + (\mu + \lambda) \nabla \mathrm{div} u_{\varepsilon}^{\varepsilon} )\mathrm{d}x\mathrm{d}t\\
	=&\int_0^T \int_\Omega \varphi (\mu |\nabla u_{\varepsilon}^{\varepsilon}|^2 + (\mu + \lambda)(\mathrm{div} u_{\varepsilon}^{\varepsilon})^2)\mathrm{d}x\mathrm{d}t + \int_0^T \int_\Omega  \nabla \varphi \cdot (\mu u_{\varepsilon}^{\varepsilon} \nabla u_{\varepsilon}^{\varepsilon}+ (\mu + \lambda) u_{\varepsilon}^{\varepsilon} \mathrm{div} u_{\varepsilon}^{\varepsilon})\mathrm{d}x\mathrm{d}t\\
	\rightarrow & \int_0^T \int_\Omega \varphi (\mu |\nabla u|^2 + (\mu + \lambda)(\mathrm{div} u)^2)\mathrm{d}x\mathrm{d}t + \int_0^T \int_\Omega  \nabla \varphi \cdot (\mu u \nabla u + (\mu + \lambda) u \mathrm{div} u)\mathrm{d}x\mathrm{d}t,
\end{aligned}
$$
as $\varepsilon \rightarrow 0$. This proves the claim for the diffusion terms.

Letting $\varepsilon$ go and to zero in \eqref{eq:thm_1}, and using \eqref{equ-} and what we have proved is that in the limit, then we complete the proof of Theorem \ref{localenergy}.
\end{proof}

\section{Global energy equality of the compressible Navier-Stokes equations with constant viscosity}
Next, we extend the result of Theorem \ref{localenergy} to the global energy equality to prove the Theorem \ref{theoremcomNS}.
For the case of a bounded domain with the no-slip boundary condition $u=0$ on $\partial\Omega$, we introduce a test function $\varphi=\psi_{\tau}\phi_{\delta}.$ We fix small constants $\tau>0, \delta>0$, and define the cut-off functions $\psi_\tau(t) \in C_0^1((\tau, T-\tau))$ and $\phi_\delta(x) \in C_0^1(\Omega)$ satisfying
$$
\left\{\begin{array}{l}
	0 \leq \phi_\delta(x) \leq 1, \quad \phi_\delta(x)=1 \text { if } x \in \Omega^{\delta}=:\{x \in \Omega \text { and } \operatorname{dist}(x, \partial \Omega) \geq \delta\} , \\
	\phi_\delta \rightarrow 1 \text { as } \delta \rightarrow 0,  \text { and }\left|\nabla \phi_\delta\right| \leq \frac{2}{\operatorname{dist}(x, \partial \Omega)}.
\end{array}\right.
$$
In view of Theorem \ref{localenergy}, we obtain
\begin{equation}\label{eq:delta}
	\begin{split}
		&\frac{1}{2}\int_0^T \int_{\Omega}\psi_{\tau}^{\prime}\phi_{\delta} \rho |u |^{2}\mathrm{d} x \mathrm{d} t+\int_0^T \int_{\Omega}\psi_{\tau}\phi_{\delta} \left[\rho^{\gamma}{\rm div}u-\mu|\nabla u|^{2}-(\mu+\lambda)({\rm div}u)^{2}\right] \mathrm{d} x \mathrm{d} t\\
		&\quad+\int_0^T \int_{\Omega}\psi_{\tau}\nabla \phi_{\delta} \cdot\left[\frac{1}{2}(\rho u) |u |^{2}+\rho^{\gamma}u -\mu u\nabla u-(\mu+\lambda)u {\rm div}u\right]\mathrm{d} x \mathrm{d} t=0.
	\end{split}
	\end{equation}

Next, we take the limit of \eqref{eq:delta} as $\delta \to 0$. Since $\phi_\delta \to 1$ pointwise and is bounded, the terms without $\nabla\phi_\delta$ converge. The terms involving $\nabla\phi_\delta$ vanish because $\nabla\phi_\delta$ is supported in $\Omega \setminus \Omega^\delta$ where the measure goes to zero. Specifically,
$$
\begin{aligned}
	&\quad \left| \int_0^T \int_{\Omega}\psi_{\tau}\nabla \phi_{\delta} \cdot\left[\frac{1}{2}(\rho u) |u |^{2}+\rho^{\gamma}u -\mu u\nabla u-(\mu+\lambda)u {\rm div}u\right]\mathrm{d} x \mathrm{d} t \right| \\
	& \leq C\|\nabla u\|_{L^{2}(L^{2})}\|\rho u\|^{2-\frac{p}{2}}_{L^{\infty}(L^{2}(\Omega \setminus\Omega^{\delta}))} \|u\|^{\frac{p}{2}}_{L^{p}(L^{q}(\Omega \setminus\Omega^{\delta}))}\\
	&\quad +C\|\nabla u\|_{L^{2}(L^{2})}(\|\rho \|_{L^{2}(L^{2}(\Omega \setminus\Omega^{\delta}))}+\|\nabla u\|_{L^{2}(L^{2}(\Omega \setminus\Omega^{\delta}))})\\
	&\rightarrow 0, \text{ as } \delta \rightarrow 0.
\end{aligned}
$$
Thus, after letting $\delta \to 0$ in \eqref{eq:delta}, it holds that
\begin{equation}\label{eq:boundedlocalenrygy}
    \frac{1}{2}\int_{0}^{T}\int_{\Omega}\psi_{\tau}^{\prime}\rho|u|^{2}\mathrm{d}x\mathrm{d}t+\int_{0}^{T}\int_{\Omega}\psi_{\tau}\rho^{\gamma}\mathrm{div}u-\psi_{\tau}\left(\mu|\nabla u|^{2}+(\mu+\lambda)(\mathrm{div}u)^{2}\right)\mathrm{d}x\mathrm{d}t=0.
\end{equation}

On the other hand, it follows from the continuity equation that
$$
	\int_{0}^{T}\int_{\Omega}\psi_{\tau}\rho^{\gamma}\mathrm{div}u \mathrm{d}x\mathrm{d}t =-\int_{0}^{T}\int_{\Omega}\psi_{\tau}\rho^{\gamma-1}(\rho_t+u\cdot \nabla \rho) \mathrm{d}x\mathrm{d}t  = \frac{1}{\gamma-1}\int_{0}^{T}\int_{\Omega}\psi_{\tau}^{\prime}\rho^{\gamma} \mathrm{d}x\mathrm{d}t.
$$
Thus, the equation \eqref{eq:boundedlocalenrygy} becomes
\begin{equation}\label{eq:boundedlocalenergy2}
     \int_{0}^{T} \int_{\Omega} \psi_{\tau}^{\prime} \left(\frac{1}{2}\rho |u|^{2} +\frac{1}{\gamma - 1}\rho^{\gamma} \right)\mathrm{d}x\mathrm{d}t   -\int_{0}^{T} \int_{\Omega} \psi_{\tau} (\mu |\nabla u|^{2} + (\mu + \lambda)| \mathrm{div} u|^{2}) \mathrm{d}x\mathrm{d}t = 0.
\end{equation}
Denote $$E(t)=\int_{0}^{T} \int_{\Omega}  \frac{1}{2}\rho |u|^{2} +\frac{1}{\gamma - 1}\rho^{\gamma} \mathrm{d}x\mathrm{d}t,\quad D(t)=\int_{0}^{t} \int_{\Omega}  \mu |\nabla u|^{2} + (\mu + \lambda)| \mathrm{div} u|^{2} \mathrm{d}x\mathrm{d}t.$$
Rewriting \eqref{eq:boundedlocalenergy2} using the definitions of $E(t)$ and $D(t)$, we get
$$
\int_0^T \psi_\tau'(t) E(t) dt + \int_0^T \psi_\tau(t) D'(t) dt = 0.
$$
Since this holds for any $\psi_\tau \in C_0^1((0,T))$, we conclude that 
\begin{equation}\label{eq:energyequalitydistribution}
    (E+D)'=0 \text{ in } \mathcal{D}'((0,T)).
\end{equation}

For the case of a \textbf{periodic domain $\Omega=\mathbb{T}^{N}$}, by choosing $\varphi=\psi_{\tau}$, using Theorem \ref{localenergy} and repeating the above proof  yields the same equality \eqref{eq:energyequalitydistribution}.

To finish all the proof, it suffices to establish the energy equality up to the initial time t = 0 by the similar method in \cite{yu2017energy}.
The one difference is to show the continuity of $\sqrt{\rho} u$ in the strong topology at $t=0$.

It is easy to see, for any $\alpha \geq \frac{1}{2}$,
\begin{equation*}
	\partial_{t}(\rho^{\alpha}) = -\alpha \rho^{\alpha} \mathrm{div} u - 2\alpha \rho^{\alpha - \frac{1}{2}} u \cdot \nabla \sqrt{\rho},
\end{equation*}
which, together with \eqref{con:defcomNS} and \eqref{condi1-th2}, implies
\begin{equation*}
	\rho^{\alpha} \in L^{\infty}(0, T; W^{1,\frac{3}{2}}(\Omega)), \quad \partial_{t}(\rho^{\alpha}) \in L^{2}(0, T; L^{\frac{12}{11}}(\Omega)).
\end{equation*}
It follows that, by the Lemma \ref{aubinlions},
\begin{equation*}
	\rho^{\alpha} \in C([0, T]; L^{r}(\Omega)), \quad r < \frac{3N}{2N-3},
\end{equation*}
so we can use  $u_{0} \in L^{\frac{6N}{6-N}}$ to deduce that
\begin{equation}\label{continuity_sqrt_rho_u}
	(\sqrt{\rho }u) (t) \rightarrow (\sqrt{\rho }u) (0) \quad \text{strongly in } L^2(\Omega) \text{ as } t \to 0^+.
\end{equation}
Following the similar manner of proof in \cite{yu2017energy}, we can complete the proof of Theorem \ref{theoremcomNS}.

\section{The energy conservation for the compressible Navier-Stokes equations with degenerate viscosity}\label{sec:deg}
Next, we consider the compressible Navier-Stokes equations with degenerate viscosity, where the density is strictly bounded away from vacuum. The systems are stated as follows:
	\begin{equation}\label{eq:comNSdeg}
		\begin{split}
			(\rho u)_t + \text{div}(\rho u \otimes u) - 2\nu \text{div}(\rho \mathbb{D} u) + \nabla P &= 0,  \\
			\rho_t + \text{div}(\rho u) &= 0,
		\end{split}
	\end{equation}
    where $\mathbb{D} u = \frac{1}{2}(\nabla u + \nabla^T u)$ is the strain tensor and the viscosity coefficients satisfy $\nu > 0$. Then, we present the definition and the energy conservation criteria of the  corresponding weak solutions.
    \begin{definition}\label{defcomNSdeg}
    The pair $(\rho, u)$ is called a global weak solution to the degenerate viscosity system \eqref{eq:comNSdeg} with initial data \eqref{con:initial} if, for any $t \in[0, T]$:
    \begin{itemize}
        \item The momentum equations \eqref{eq:comNSdeg} hold in $\mathcal{D}^{\prime}((0, T) \times \Omega)$ satisfying:
        $$
        \begin{gathered}
            \rho \geq 0, \quad \rho \in L^{\infty}\left(0, T ; L^\gamma(\Omega)\right),\quad \nabla \sqrt{\rho} \in L^{\infty}\left(0, T ; L^\frac{3}{2}(\Omega)\right)\\
            \sqrt{\rho} u \in L^{\infty}\left(0, T ; L^2(\Omega)\right), \quad \sqrt{\rho} \nabla u \in L^2\left(0, T ; L^2(\Omega)\right);
        \end{gathered}
        $$
        
        \item The initial conditions \eqref{con:initial} hold in $\mathcal{D}^{\prime}(\Omega)$;

        \item $(\rho, u)$ is a renormalized solution of the continuity equation in the sense of DiPerna-Lions \cite{diperna1989ordinary};
        
        \item The energy inequality holds for almost every $t \in [0,T]$:
        \begin{equation}\label{eq:energy_comNSdeg}
            \int_{\Omega}\left(\frac{1}{2} \rho|u|^2+\frac{\rho^\gamma}{\gamma-1}\right) \mathrm{d} x+\int_0^t \int_{\Omega} \rho|\mathbb{D} u|^2 \mathrm{d} x \mathrm{d} s \leq \int_{\Omega}\left(\frac{1}{2} \rho_0\left|u_0\right|^2+\frac{\rho_0^\gamma}{\gamma-1}\right) \mathrm{d} x.
        \end{equation}
    \end{itemize}
\end{definition}
\begin{remark}
	The condition $\nabla\sqrt{\rho}\in L^{\infty}L^{\frac{3}{2}}$ is reasonable in this context, as it is consistent with the known existence theory for weak solutions. Specifically, for the degenerate compressible Navier–Stokes equations, global existence results \cite{vasseur2016existence} ensure that $\nabla \rho^{\frac{\gamma}{2}}\in L^{2}L^{2}$ and $\nabla\sqrt{\rho}\in L^{\infty}L^{2}$.
\end{remark}

\begin{theorem}\label{theoremcomNSdeg}
    Let $(\rho, u)$ be a weak solution of \eqref{eq:comNSdeg} in the sense of Definition \ref{defcomNSdeg}. Assume that the density satisfies
    \begin{equation}\label{condi1-th2}
        0 < \underline{\rho} \leq \rho(t, x)\leq \bar{\rho} < \infty.
    \end{equation}
     and the initial condition satisfies
    \begin{equation}\label{condi2-th2}
        \sqrt{\rho_0} u_0 \in L^{\frac{4N}{N+2}}(\Omega).
    \end{equation}
    If the velocity satisfies
    \begin{equation}\label{con:shinbrot_cond2}
        u \in L^p\left(0, T ; L^q(\Omega)\right) \quad \text{with} \quad \begin{cases}
            \frac{1}{p}+\frac{3}{q} \leq 1,\quad\mbox{if}  \quad 3\leq q < 4,\\
            \frac{2}{p}+\frac{2}{q} \leq 1,\quad\mbox{if}  \quad 4\leq q \leq\infty,
        \end{cases}
    \end{equation}
    then the energy equality \eqref{eq:energy_comNSdeg} holds for any $t \in[0, T]$.
\end{theorem}
\begin{remark}
    Different from Theorem \ref{theoremcomNS}, we have obtained the energy conservation under the assumption $u\in L^{2}L^{\infty}$, which is one of the endpoints. 
\end{remark}
Finally, let us provide the proof of Theorem \ref{theoremcomNSdeg}. We can modify the proof in Section 3 and 4 slightly to arrive the case $(p,q)=(4,4)$. Note that, for any weak solution ($\rho, u$), condition \eqref{condi1-th2} implies that

$$
\|u\|_{L^{\infty}\left(0, T ; L^2(\Omega)\right)} \leq C<\infty, \quad\|\nabla u\|_{L^2\left(0, T ; L^2(\Omega)\right)} \leq C<\infty .
$$
The one difference is to show that 
\begin{equation}\label{eq:degeneratedifussion}
	 \lim_{\varepsilon \to 0} \int_0^T \int_\Omega \varphi  u_{\varepsilon}^\varepsilon \operatorname{div}(\rho \mathbb{D} u)_{\varepsilon}^{\varepsilon} \mathrm{d}x\mathrm{d}t 
	 = -\int_0^T \int_{\Omega} \varphi \rho|\mathbb{D} u|^2 + \rho u\otimes \nabla\varphi:\mathbb{D} u\  \mathrm{d}x\mathrm{d}t.
\end{equation}
For this term, we follow the similar manner of proof in \cite{yu2017energy} to prove it.
The another difference is to show  the continuity of $\sqrt{\rho} u$ in the strong topology at $t=0$. 
Applying a similar argument in Section 4, we can show
$$
\sqrt{\rho} \in C\left([0, T] ; L^r(\Omega)\right),\quad (r<\frac{2N}{N-2}),
$$
so we need $\sqrt{\rho_{0}}u_0 \in L^k$ where $k>\frac{4N}{N+2}$ or $u_{0} \in L^{m}$ with $m > N$ to get \eqref{continuity_sqrt_rho_u}.

As for $p\neq 4$, we have embedding as follows,
$$
L^{p}(0,T;L^{q}(\Omega))\cap L^{2}(0,T;H^{1}(\Omega))\hookrightarrow L^{4}(0,T;L^{4}(\Omega)),\quad\mbox{where}\quad \frac{1}{p}+\frac{3}{q}=1,
$$
for $4<p\leq\infty$, and
$$
L^{p}(0,T;L^{q}(\Omega))\cap L^{\infty}(0,T;L^{2}(\Omega))\hookrightarrow L^{4}(0,T;L^{4}(\Omega)),\quad\mbox{where}\quad \frac{2}{p}+\frac{2}{q}=1,
$$
for $4<p\leq\infty$, which finish the proof.

\section*{Acknowledgments}
The corresponding author is supported by the Fundamental Research Funds for the Central Universities, Project No.3072025CFJ2406. The authors would like to thank the referee for their careful reading and useful suggestions and comments.

\section*{Declarations}

\textbf{Data availability} We do not analyse or generate any datasets, because our work proceeds within a theoretical approach.

\noindent\textbf{Conflict of interest} On behalf of all authors, the corresponding author states that there is no conflict of interest.

\bibliography{reference}

\end{document}